\documentclass[10pt,a4paper]{article}
\usepackage{amsmath,amsthm,amssymb,amscd}
\usepackage{graphicx,graphics}
\usepackage{verbatim}
\usepackage{mathrsfs}
\usepackage{enumerate}
\usepackage{color}
\usepackage[top=1.2in, bottom=1.2in, left=1.2in, right=1.2in]{geometry}

\newtheorem{theorem}{Theorem}[section]
\newtheorem{proposition}[theorem]{Proposition}

\newtheorem{lemma}[theorem]{Lemma}

\newtheorem{corollary}[theorem]{Corollary}
\newtheorem{claim}{Claim}[theorem]
\newtheorem{conjecture}[theorem]{Conjecture}
\newtheorem{observation}[theorem]{Observation}

\begin{document}

\title{Unions of 1-factors in $r$-graphs and overfull graphs}

\author{Ligang Jin\thanks{
	   Department of Mathematics,
       Zhejiang Normal University,
       Yingbin Road 688,
       321004 Jinhua,
       China;
       Grant Numbers: NSFC 11801522 and Qianjiang Talent Program of Zhejiang Province QJD1803023;    
       ligang.jin@zjnu.cn},
       Eckhard Steffen\thanks{
       	Paderborn Center for Advanced Studies and
				Institute for Mathematics,
       	Paderborn University,
       	Warburger Str. 100,
       	33098 Paderborn,
       	Germany;	es@upb.de}}

\maketitle

\begin{abstract}
{\small 
We prove lower bounds for the fraction of edges of an $r$-graph which can be covered by the union of 
$k$ 1-factors. The special case $r=3$ yields some known results for cubic graphs. Furthermore, we introduce the concept of
$k$-overfull-free $r$-graphs and achieve better bounds for these graphs.}  
\end{abstract}

\par\bigskip\noindent
\textbf{Keywords}: $r$-graphs; 1-factors; overfull graphs

\section{Introduction}
We consider finite graphs $G$ with vertex set $V(G)$ and edge set $E(G)$. Graphs do not contain loops in this paper. 
For $v,w \in V(G)$, the number of edges 
between $v$ and $w$ is denoted by $\mu(v,w)$ and $\mu(G) = \max \{\mu(v,w)\colon\ v,w \in V(G)$\}. $\mu(v,w)$ is called the multiplicity of $vw$ and $\mu(G)$ the multiplicity of $G$. A graph is simple if $\mu(v,w) \leq 1$ for any two vertices $v,w$.
The number of edges which are incident to vertex $v$ is the vertex degree of $v$ which is denoted by $d_G(v)$. 
The maximum vertex degree of $G$ is $\max \{d_G(v) : v \in V(G)\}$ and it is denoted by $\Delta(G)$. Further
$\delta(G)$ denotes the minimum degree of a vertex of $G$. 

\subsection{1-factor covering}
The following celebrated conjecture, often referred to as the Berge-Fulkerson conjecture, is due to Fulkerson and appears first in \cite{Fulkerson1971168}:
\begin{conjecture}[Berge-Fulkerson conjecture \cite{Fulkerson1971168}]
	Every bridgeless cubic graph $G$ has six 1-factors such that each edge of $G$ is contained in precisely two of them.	
\end{conjecture}
A set of such six 1-factors in the conjecture is called a Fulkerson cover of $G$.
It is straightforward that Berge-Fulkerson Conjecture implies the existence of five 1-factors whose union is the edge-set of the graph $G$.
This naturally raises a seemly weaker conjecture, attributed to Berge (unpublished, see e.g. \cite{CQZhang1997}).

\begin{conjecture}[Berge conjecture] \label{Berge_conj}
	Every bridgeless cubic graph $G$ has five 1-factors such that each edge of $G$ is contained in at least one of them.
\end{conjecture}

A set of the five 1-factors in Berge Conjecture is called a Berge cover of $G$.
Recently, Mazzuoccolo \cite{Mazzuoccolo2011} proved that the previous two conjectures are equivalent. 
It is unclear whether the same equivalence holds for every single bridgeless cubic graph, in other words, does a graph having a Berge cover always have a Fulkerson cover?

Let $r$ be a positive integer.
A graph $G$ is $r$-regular, if $d_G(v)=r$ for all $v \in V(G)$.
Let $X \subseteq V(G)$  be a set of vertices. The subgraph of $G$ induced by $X$ is denoted by $G[X]$, and the set of edges with precisely one end in $X$ by $\partial_G(X)$. 
An $r$-regular graph $G$ is an $r$-graph if $|\partial_G(X)| \geq r$ for every odd set $X \subseteq V(G)$.

A cubic graph is a 3-graph if and only if it is bridgeless. Moreover, it was proved in \cite{Seymour_1979} that every $r$-graph has a 1-factor.
Hence, it is natural to consider similar questions on perfect matching covering for $r$-graphs as for bridgeless cubic graphs. 
In particular, aforementioned two conjectures were generalized to $r$-graphs.
In 1979, Seymour \cite{Seymour_1979} proposed the generalized Berge-Fulkerson conjecture:

\begin{conjecture}[Generalized Berge-Fulkerson conjecture \cite{Seymour_1979}] \label{conj_GBF}
Every $r$-graph has $2r$ 1-factors such that each edge is contained in precisely two of them.
\end{conjecture}

Trivially, this conjecture implies the following generalized form of Conjecture \ref{Berge_conj}, 
first proposed by Mazzuoccolo \cite{Mazzuoccolo201373}.

\begin{conjecture}[Generalized Berge conjecture \cite{Mazzuoccolo201373}]\label{conj_excessive_index}
Every $r$-graph $G$ has $2r-1$ 1-factors such that each edge is contained in at least one of them.
\end{conjecture}

The value $2r-1$ in the conjecture is best possible, that is, it can not be smaller, as shown in \cite{Mazzuoccolo201373}. In the same paper, Mazzuoccolo proved the equivalence between the generalized Berge-Fulkerson conjecture and the generalized Berge conjecture, in a similar way as he did for cubic case.

The excessive index $\chi'_e(G)$ of a graph $G$ is the minimum number of 1-factors needed to cover $E(G)$. This parameter, also called the perfect matching index in \cite{FouquetVanherpe}, was widely studied in the literature, e.g., \cite{Bonisoli2007, Cariolaro, Mazzuoccolo201373, Mazzuoccolo2014264, MazzuoccoloYoung, Rajasingh}.
It is reasonable to consider the excessive index for $r$-graphs in the context that it can be arbitrary large for some family of bridgeless $r$-regular graphs, constructed in \cite{MazzuoccoloYoung}.
However, it is an open question whether there exists a constant $k$ such that $\chi'_e(G)\leq k$ for all $r$-graphs $G$ for any fixed $r\geq 3$. 
The result of Mazzuoccolo \cite{Mazzuoccolo201373} shows that if such $k$ exists then it is at least $2r-1$. The generalized Berge conjecture asserts that such $k$ exists and $k = 2r-1$.  

Partial covers of $r$-graphs with 1-factors are of great interest, see e.g.~\cite{Jin_Steffen_2017, Steffen_2015}.
In this paper, we consider the following relaxed form of the generalized Berge conjecture: Over all $r$-graphs $G$ for any 
fixed $r$, what is the maximum constant $c~(c\leq 1)$, such that $G$ has $2r-1$ 1-factors whose union contains at least $c|E(G)|$ edges? Note that the generalized Berge conjecture asserts that $c=1$. 
We will show that $c\geq 1-e^{-2}\approx 0.8647$. We will also show a second lower bound for $c$ which depends on $r$, 
but which is always greater than $1-e^{-2}$.
In fact, this second lower bound is an approximation to the following more general problem.

Given an $r$-graph $G$, let $\mathcal{M}$ be the set of distinct 1-factors in $G$.
Fix a positive integer $k$. Define
$$m(r,k,G)=\max_{M_1,\ldots,M_k\in \mathcal{M}} \frac{|\bigcup_{i=1}^k M_i|}{|E(G)|},$$
and
$$m(r,k)=\inf_G m(r,k,G),$$
where the infimum is taken over all $r$-graphs.
Clearly, $m(r,k)\leq m(r,k+1) \leq 1.$
With this notation, the generalized Berge conjecture can be reformulated as follows:
\begin{conjecture}\label{}
	$m(r,2r-1)= 1$ for every integer $r$ with $r\geq 3$.
\end{conjecture}

The parameter $m(r,k)$ has primarily been studied in cubic case, i.e.~$r=3$.
Berge's conjecture states that $m(3,5)=1$.
Kaiser, Kr\'{a}l and Norine \cite{Kaiser2006} proposed a lower bound for $m(3,k)$ as \begin{equation}\label{LBoundCubic}
m(3,k)\geq 1-\prod_{i=1}^k \frac{i+1}{2i+1},
\end{equation}
and verified it for the case $k\in \{2,3\}$. Meanwhile,
Patel \cite{Patel} conjectured that $m(3,2)=\frac{3}{5}$, $m(3,3)=\frac{4}{5}$ and
$m(3,4)=\frac{14}{15}$. 
Since the example of Petersen graph, the result of Kaiser, Kr\'{a}l and Norine confirms that $m(3,2)=\frac{3}{5}$. But the exact values for $m(3,3)$ and $m(3,4)$ are still unknown. 
A complete proof for the lower bound in (\ref{LBoundCubic}) was later given by Mazzuoccolo \cite{Mazzuoccolo2013313}.

In Section \ref{sec:covering}, we obtain the following lower bound for $m(r,k)$:
\begin{equation}\label{BoundEven}
m(r,k)\geq 1-\prod_{i=1}^k \frac{(r^2-3r+1)i-(r^2-5r+3)}{(r^2-2r-1)i-(r^2-4r-1)}
\end{equation} 
for any even $r\geq 4$ and any $k\geq 1$, and
\begin{equation}\label{BoundOdd}
m(r,k)\geq 1-\prod_{i=1}^k \frac{(r^2-2r-1)i-(r^2-4r+1)}{(r^2-r-2)i-(r^2-3r-2)}
\end{equation} for any odd $r\geq3$ and any $k\geq 1$.

For instance of small $r$ and $k$, the values of this lower bound are listed in Table \ref{table}.
\begin{table}[hh] \label{table}
	\centering
	\begin{tabular}{|c|l|l|l|l|}
		\hline
		& $r=3$ & $r=4$ & $r=5$\\
		\hline
		$m(r,2)\geq$ & $0.6\geq 0.5556$ & $0.45 \geq 0.4375$ & $ 0.3714 \geq 0.36$ \\
		$m(r,3)\geq$ & $ 0.7714\geq 0.7037$ & $0.6 \geq 0.5781$ & $ 0.5081 \geq 0.488$  \\
		$m(r,4)\geq$ & $0.873 \geq 0.8025$ & $0.7103 \geq 0.6836$ & $ 0.6157 \geq 0.5904$ \\
		$m(r,5)\geq$ & ${\bf 0.9307\geq 0.8683}$ &  $0.7908\geq 0.7627$ & $ 0.7 \geq 0.6723$ \\
		$m(r,6)\geq$ & $0.9627  \geq 0.9122$ & $ 0.8492 \geq 0.822$ & $ 0.766 \geq 0.7379$ \\
		$m(r,7)\geq$ & $0.9801 \geq 0.9415$ & ${\bf 0.8914 \geq 0.8665}$ & $0.8176 \geq 0.7903$ \\
		$m(r,8)\geq$ & $0.9895 \geq 0.961$ & $0.9219 \geq 0.8999$  & $0.8578 \geq 0.8322$ \\
		$m(r,9)\geq$ & $0.9945 \geq 0.974$ & $0.9439 \geq 0.9249$  & ${\bf 0.8892 \geq 0.8658}$ \\
		\hline
	\end{tabular}
	\caption{Approximate values of the two lower bounds for $m(r,k)$ presented in formulations (\ref{BoundEven}) and (\ref{BoundOdd}) and in Theorem 3.1, shown respectively in the left and the right sides of the inequality in the table. 
		In particular, the one for $m(r,2r-1)$ is presented in bold.}
\end{table}

In particular, if we take $r=3$, this lower bound coincides with the established bound in (\ref{LBoundCubic}); if we take $k=2r-1$, it gives a partial result to the generalized Berge conjecture, and the approximate value of $m(r,2r-1)$ is shown in bold in Table \ref{table}.

Now we are going to show that the lower bounds in (\ref{BoundEven}) and (\ref{BoundOdd}) is always better than $1-e^{-2}$.
Let $f(x)$ be a function defined by 
$$f(x)=\frac{(r^2-3r+1)x-(r^2-5r+3)}{(r^2-2r-1)x-(r^2-4r-1)}.$$
We can calculate that $f(1)=\frac{r-1}{r}$ and the derivative $$f'(x)=\frac{-2}{[(r^2-2r-1)x-(r^2-4r-1)]^2}<0.$$
Hence, $f(i)\leq \frac{r-1}{r}$ for each $i\geq 1$.
So when we take $k=2r-1$, the lower bound in (\ref{BoundEven}) is greater than
$$1-(\frac{r-1}{r})^{2r-1}>1-e^{-2},$$
where the last inequality is given by Corollary \ref{cor_ConstantBound}. 
Similarly, when $k=2r-1$, we can deduce that the lower bound in (\ref{BoundOdd}) is greater than $1-e^{-2}$ as well.

\subsection{Edge-colorings and overfull graphs}

A graph $G$ is $k$-overfull if $|V(G)|$ is odd, $\Delta(G) \leq k$ and $\frac{|E(G)|}{\lfloor \frac{1}{2}|V(G)| \rfloor} > k$.
It is easy to see that $G$ is $k$-overfull if and only if $2|E(G)| > k(|V(G)|-1)$. Furthermore, the 
$k$-deficiency of $G$ is $k|V(G)| - 2|E(G)|$, and it is denoted by $s_k(G)$.  

A $k$-edge-coloring of $G$ is a mapping $c : E(G) \rightarrow \{1, \dots, k\}$ such that adjacent edges are colored differently.
The chromatic index $\chi'(G)$ is the minimum number $k$ such that $G$ has a $k$-edge-coloring. Vizing \cite{Vizing_1964} 
proved that $\Delta(G)\leq \chi'(G)\leq \Delta(G) + \mu(G)$, in particular if $G$ is simple,
then $\chi'(G) \in \{\Delta(G), \Delta(G) + 1\}$. We say that $G$ is class 1 if $\chi'(G) = \Delta(G)$, and it is class 2 otherwise. 

Clearly, $\Delta(G)$ is a lower bound for the chromatic index of $G$. Overfull graphs are class 2 graphs for the trivial reason
that they contain too many edges. In general we have that 
$\chi'(G) \geq \max_{H \subseteq G} \lceil \frac{|E(H)|}{\lfloor \frac{1}{2}|V(H)| \rfloor}\rceil$.

A graph $G$ is critical with respect to $\chi'(G)$, if $\chi'(G-e) < \chi'(G)$ for every $e \in E(G)$. For simple graphs we have the definition
of a $k$-critical graph which says that a critical graph $H$ is $k$-critical, if $\Delta(H) = k$ and $\chi'(H)=k+1$.  
Vizing \cite{Vizing_1965} proved the classical result that a simple class 2 graph with maximum degree $k$ contains a $t$-critical subgraph 
for every $t \in \{2, \dots, k\}$. These results are the motivation for the result of Section \ref{Overfull graphs}, which proves that a $k$-overfull graph contains a $t$-overfull subgraph for every $t \in \{2, \dots, k\}$.

A graph $G$ is $k$-overfull-free, if it does not contain a $k$-overfull subgraph. Clearly, there are no 1-critical graphs and the 
2-critical graphs are the odd circuits which are also the connected 2-overfull graphs. Hence we have: A graph is 2-overfull-free
if and only if it is bipartite. We study $k$-overfull-free graphs in Section \ref{Overfull graphs}.
If an $r$-regular graph $G$ is class 1, then surely $G$ is an $r$-graph and $G$ is $r$-overfull-free. For $i \in \{1,2\}$ let
$G_i$ be an $r_i$-regular graph. We say that an $r$-graph $G$ is decomposable into $G_1$ and $G_2$ if $r=r_1 + r_2$, $V(G_i)=V(G)$
and $E(G) = E(G_1) \cup E(G_2)$. 
We will characterize some decomposable $r$-graphs in terms of excluded overfull subgraphs.  

\section{The perfect matching polytope}
Let $G$ be a graph and $w$ be a vector of $\mathbb{R}^{E(G)}$.
The entry of $w$ corresponding to an edge $e$ is denoted by $w(e)$, and for $A\subseteq E$, we define $w(A)=\sum_{e\in A}w(e)$.
The vector $w$ is a \emph{fractional 1-factor}  if it satisfies
\begin{enumerate}[(i)]
  \item $0\leq w(e)\leq 1$ for every $e\in E(G)$, and
  \item $w(\partial (\{v\}))=1$ for every $v\in V(G)$, and
  \item $w(\partial (S))\geq1$ for every $S\subseteq V(G)$ with odd cardinality.
\end{enumerate}

Let $F(G)$ denote the set of all fractional 1-factors of a graph $G$.
If $M$ is a 1-factor, then its characteristic vector $\chi^M$ is contained in $F(G)$. Furthermore, if $w_1,\ldots,w_n\in F(G)$, then any convex
combination $\sum_{i=1}^n\alpha_iw_i$ (where $\alpha_1,\ldots,\alpha_n$ are nonnegative real numbers summing up to 1) also belongs to $F(G)$.
It follows that $F(G)$ contains the
convex hull of all the vectors $\chi^M$ where $M$ is a 1-factor of $G$.
The following theorem by Edmonds asserts that the converse inclusion also
holds:
\begin{theorem}[Perfect Matching Polytope Theorem \cite{Edmonds1965}]\label{thm_Edmonds}
For any graph $G$, the set $F(G)$ coincides with the convex hull of the characteristic vectors of all 1-factors of $G$.
\end{theorem}

Towards the generalized Berge-Fulkerson conjecture, Seymour \cite{Seymour_1979} gave an alternative proof of the following theorem, which is a corollary of Edmonds's matching polytope theorem (see \cite{Seymour_1979} for the details between these two theorems).
\begin{theorem}[\cite{Seymour_1979}]\label{thm_multicoloring}
	For any $r$-graph $G$, there is a positive integer $p$ such that $G$ has $rp$ 1-factors and each edge is contained in precisely $p$ of them.
\end{theorem}

We will use this theorem to deduce our first lower bound in the next section. Moreover, the following property on fractional 1-factors will play a crucial role in the proof for our second lower bound.
\begin{lemma} [\cite{Kaiser2006}] \label{lem_FPMtoPM}
Let $w$ be a fractional 1-factor of a graph $G$ and $c\in \mathbb{R}^{E(G)}$.
Then $G$ has a 1-factor $M$ such that $c\cdot \chi^M \geq c\cdot w$, where $\cdot$ denotes the scalar product,
and $|M\cap C|=1$ for each edge-cut $C$ of odd cardinality and with $w(C)=1$.
\end{lemma}
The proof of this lemma was given in \cite{Kaiser2006}, where Theorem \ref{thm_Edmonds} is the main tool for the proof.

\section{Lower bounds for $m(r,k)$}\label{sec:covering}
We are going to deduce a lower bound for the parameter $m(r,k)$ by using Theorem \ref{thm_multicoloring} only.
\begin{theorem}\label{thm_trivial low bound}
$m(r,k)\geq 1-(\frac{r-1}{r})^k$ for every positive integers $r$ and $k$ with $r\geq 3$.
\end{theorem}

\begin{proof}
(induction on $k$.) Since every $r$-graph has a 1-factor, which covers a fraction $\frac{1}{r}$ of the edges, the proof is trivial for $k=1$.
We proceed to the induction step.
Let $G$ be any $r$-graph and $E=E(G)$.
By the induction hypothesis, $G$ has $k-1$ many 1-factors $M_1,\ldots,M_{k-1}$ such that
\begin{equation}\label{eqn_1st bound 1}
\frac{|\bigcup_{i=1}^{k-1}M_i|}{|E|}\geq 1-(\frac{r-1}{r})^{k-1}.
\end{equation}
Moreover, by Theorem \ref{thm_multicoloring}, there exists a positive integer $p$ such that $G$ has $rp$ 1-factors $F_1,\ldots,F_{rp}$ and each edge is contained in precisely $p$ of them.
It follows that for every $X\subseteq E$, graph $G$ has a 1-factor $F$ among $F_1,\ldots,F_{rp}$ such that $|F\cap X|\geq \frac{|X|}{r}$.
In particular, let $X=E\setminus \bigcup_{i=1}^{k-1}M_i$ and consequently, take $M_k=F$. Thus,
\begin{equation}\label{eqn_1st bound 2}
|M_k\cap (E\setminus \bigcup_{i=1}^{k-1}M_i)|\geq \frac{|E\setminus \bigcup_{i=1}^{k-1}M_i|}{r},
\end{equation}
that is,
\begin{equation}\label{eqn_1st bound 3}
\frac{|\bigcup_{i=1}^{k}M_i|-|\bigcup_{i=1}^{k-1}M_i|}{|E|} \geq \frac{1}{r}(1-\frac{|\bigcup_{i=1}^{k-1}M_i|}{|E|}).
\end{equation}
It follows that
\begin{equation}\label{eqn_1st bound 4}
	\frac{|\bigcup_{i=1}^{k}M_i|}{|E|}\geq (1-\frac{1}{r})\frac{|\bigcup_{i=1}^{k-1}M_i|}{|E|}+\frac{1}{r}\geq 1-(\frac{r-1}{r})^{k}
\end{equation}
where the last inequality follows by using the inequality (\ref{eqn_1st bound 1}).
Therefore, $m(r,k,G)\geq 1-(\frac{r-1}{r})^k$ and by the choice of $G$, we have $m(r,k)\geq 1-(\frac{r-1}{r})^k$.
\end{proof}

In particular, if we take $k=2r-1$, we can further deduce from this theorem a constant lower bound for $m(r,2r-1)$.
\begin{corollary} \label{cor_ConstantBound}
	For every integer $r\geq 3$, we have $m(r,2r-1)\geq 1-e^{-2} \approx 0.8647.$
\end{corollary}
\begin{proof}
	Let $f(r)$ denote the function $1-(\frac{r-1}{r})^{2r-1}$. It is easy to see that $f(r)$ is strictly monotonic decreasing with respect to $r$. Moreover, $\lim\limits_{r\to +\infty}f(r)=1-e^{-2}.$ It follows with Theorem \ref{thm_trivial low bound} that $m(r,2r-1)\geq f(r)\geq 1-e^{-2}$.
\end{proof}

We now prove the following theorem, which will be used to deduce a second lower bound for $m(r,k)$.
An $i$-cut of a graph $G$ is an edge cut of $G$ of cardinality $i$.
The proof of the theorem is conducted by induction.
In the induction step, we apply Lemma \ref{lem_FPMtoPM} to a well-chosen fractional 1-factor, whose existence can be guaranteed by both inclusions of the induction hypothesis, one on the union of 1-factors and the other on $i$-cuts. 
The resulting 1-factor with its properties described in Lemma \ref{lem_FPMtoPM} and the 1-factors given by the induction hypothesis together complete the proof.
\begin{theorem}\label{thm:main1}
Let $G$ be an $r$-graph, and $V=V(G)$ and $E=E(G)$.
\begin{enumerate}[(a)]
  \itemsep=0cm
  \item If $r$ is even and $r\geq 4$, then for any positive integer $k$, graph $G$ has $k$ 1-factors $M_1,\ldots,M_k$ such that
      $$\frac{|\bigcup_{i=1}^k M_i|}{|E|}\geq 1-\prod_{i=1}^k \frac{(r^2-3r+1)i-(r^2-5r+3)}{(r^2-2r-1)i-(r^2-4r-1)}$$ and $\sum_{i=1}^k\chi^{M_i}(C)\leq (r-1)k+2$ for each $(r+1)$-cut $C$.

  \item If $r$ is odd and $r\geq 3$, then for any positive integer $k$, graph $G$ has $k$ 1-factors $M_1,\ldots,M_k$ such that
      $$\frac{|\bigcup_{i=1}^k M_i|}{|E|}\geq 1-\prod_{i=1}^k \frac{(r^2-2r-1)i-(r^2-4r+1)}{(r^2-r-2)i-(r^2-3r-2)},$$
      $\sum_{i=1}^k\chi^{M_i}(C)=k$ for each $r$-cut $C$ and $\sum_{i=1}^k\chi^{M_i}(D)\leq rk+2$ for each $(r+2)$-cut $D$.
\end{enumerate}
\end{theorem}

\begin{proof}
(induction on $k$).   

\textbf{Statement ($a$).} The statement holds for $k=1$, since the required $M_1$ can be an arbitrary 1-factor of $G$. Assume that $k\geq2.$
By the induction hypothesis, $G$ has $k-1$ many 1-factors $M_1,\ldots,M_{k-1}$ such that
$$\frac{|\bigcup_{i=1}^{k-1} M_i|}{|E|}\geq 1-\prod_{i=1}^{k-1} \frac{(r^2-3r+1)i-(r^2-5r+3)}{(r^2-2r-1)i-(r^2-4r-1)}$$
and
\begin{equation}\label{eqn_r+1 cut}
\sum_{i=1}^{k-1}\chi^{M_i}(C)\leq (r-1)(k-1)+2
\end{equation}
for each $(r+1)$-cut $C$.

For $e\in E$, let $n(e)$ denote the number of 1-factors among $M_1,\ldots,M_{k-1}$ which contain $e$, and define
\begin{equation*}\label{eq_w_i(e)}
w_k(e)=\frac{(r-2)k-(r-4)-n(e)}{(r^2-2r-1)k-(r^2-4r-1)}.
\end{equation*}

We claim that $w_k$ is a fractional 1-factor of $G$, that is, $w_k\in F(G)$. Since $k\geq 2, r\geq 4$ and $0 \leq n(e)\leq k-1$,
we can deduce that $\frac{1}{r+3}<w_k(e)<1$. Moreover,
note that for every $X\subseteq E$, the equality $\sum_{e\in X}n(e)=\sum_{i=1}^{k-1}\chi^{M_i}(X)$ always holds and so
\begin{equation}\label{eq_w_i(X)}
w_k(X)=\sum_{e\in X}w_k(e)=\frac{[(r-2)k-(r-4)]|X|-\sum_{i=1}^{k-1}\chi^{M_i}(X)}{(r^2-2r-1)k-(r^2-4r-1)}.
\end{equation}
Thus for $v\in V$, since $\sum_{i=1}^{k-1}\chi^{M_i}(\partial (\{v\}))=k-1$,
we have $w_k(\partial (\{v\}))=\frac{[(r-2)k-(r-4)]r-(k-1)}{(r^2-2r-1)k-(r^2-4r-1)}=1$.
Finally, let $S\subseteq V$ with odd cardinality. Since $G$ is an $r$-graph, we have $|\partial (S)|\geq r$.
On the other hand, by recalling that $w_k(e)>\frac{1}{r+3}$ for each edge $e$, we have $w_k(\partial (S))>1$ provided by $|\partial (S)|\geq r+3$.
Hence, we may next assume that $|\partial (S)|=r+1$ by parity. Since in this case $S$ is a $(r+1)$-cut, the formula $(\ref{eqn_r+1 cut})$ implies
$\sum_{i=1}^{k-1}\chi^{M_i}(\partial (S))\leq (r-1)(k-1)+2$,
and thus with the help of the formula $(\ref{eq_w_i(X)})$, we deduce $w_k(\partial(S))\geq \frac{[(r-2)k-(r-4)](r+1)-[(r-1)(k-1)+2]}{(r^2-2r-1)k-(r^2-4r-1)}=1$.
This completes the proof of the claim.

By Lemma \ref{lem_FPMtoPM}, the graph $G$ has a 1-factor $M_k$ such that $$(1-\chi^{\bigcup_{i=1}^{k-1}M_i})\cdot \chi^{M_k} \geq (1-\chi^{\bigcup_{i=1}^{k-1}M_i})\cdot w_k.$$
Since the left side is just
$|\bigcup_{i=1}^{k}M_i|-|\bigcup_{i=1}^{k-1}M_i|$
and the right side equals to $\frac{(r-2)k-(r-4)}{(r^2-2r-1)k-(r^2-4r-1)}(|E|-|\bigcup_{i=1}^{k-1}M_i|),$
it follows that
$$|\bigcup_{i=1}^kM_i|\geq \frac{(r^2-3r+1)k-(r^2-5r+3)}{(r^2-2r-1)k-(r^2-4r-1)}|\bigcup_{i=1}^{k-1}M_i|+\frac{(r-2)k-(r-4)}{(r^2-2r-1)k-(r^2-4r-1)}|E|,$$ which leads to
$$\frac{|\bigcup_{i=1}^k M_i|}{|E|}\geq 1-\prod_{i=1}^k \frac{(r^2-3r+1)i-(r^2-5r+3)}{(r^2-2r-1)i-(r^2-4r-1)},$$ as desired.

Moreover, let $C$ be an edge cut with cardinality $r+1$.
Clearly, $\chi^{M_k} (C)\leq r+1$.
Thus, if $\sum_{i=1}^{k-1}\chi^{M_i}(C)\leq (r-1)(k-1)$ then
$\sum_{i=1}^k\chi^{M_i}(C)\leq (r-1)k+2$, as desired.
By the formula $(\ref{eqn_r+1 cut})$ and by parity, we may next assume that $\sum_{i=1}^{k-1}\chi^{M_i}(C)=(r-1)(k-1)+2$.
In this case, we calculate from the formula $(\ref{eq_w_i(X)})$ that $w_k(C)=1$. Thus $\chi^{M_k}(C)=1$ by Lemma \ref{lem_FPMtoPM}, which yields $\sum_{i=1}^k\chi^{M_i}(C)= (r-1)k-r+4 <(r-1)k+2$, as desired.
This completes the proof of Statement ($a$).

\textbf{Statement ($b$).} We follow a similar way to prove this statement as we did for Statement $(a)$. Let $w_1$ be a vector of $\mathbb{R}^E$ defined by $w_1(e)=\frac{1}{r}$ for $e\in E$. Clearly, $w_1\in F(G)$.
By Lemma \ref{lem_FPMtoPM}, $G$ has a 1-factor $M_1$ such that $\chi^{M_1}(C)=1$ for each edge cut $C$ with odd cardinality and with $w_1(C)=1$, that is, for each $r$-cut $C$.
Therefore, the statement is true for $k=1$.

Assume $k\geq2.$
By the induction hypothesis, $G$ has $k-1$ many 1-factors $M_1,\ldots,M_{k-1}$ such that
\begin{equation*}
\frac{|\bigcup_{i=1}^{k-1} M_i|}{|E|}\geq 1-\prod_{i=1}^{k-1} \frac{(r^2-2r-1)i-(r^2-4r+1)}{(r^2-r-2)i-(r^2-3r-2)},
\end{equation*}
and for each $r$-cut $C$
\begin{equation}\label{eqn_r cut}
\sum_{i=1}^{k-1}\chi^{M_i}(C)=k-1,
\end{equation}
and for each $(r+2)$-cut $D$
\begin{equation}\label{eqn_r+2 cut}
\sum_{i=1}^{k-1}\chi^{M_i}(D)\leq r(k-1)+2.
\end{equation}

For $e\in E$, let $n(e)$ denote the number of 1-factors among $M_1,\ldots,M_{k-1}$ that contains $e$, and define
\begin{equation*}\label{b_eq_w_i(e)}
w_k(e)=\frac{(r-1)k-(r-3)-2n(e)}{(r^2-r-2)k-(r^2-3r-2)}.
\end{equation*}

We claim that $w_k\in F(G)$. Since $k\geq 2, r\geq 3$ and $0 \leq n(e)\leq k-1$,
we can deduce that $0<\frac{1}{r+4}<w_k(e)<1$. Moreover,
note that for every $X\subseteq E$, the equality $\sum_{e\in X}n(e)=\sum_{i=1}^{k-1}\chi^{M_i}(X)$ always holds and so
\begin{equation}\label{b_eq_w_i(X)}
w_k(X)=\frac{[(r-1)k-(r-3)]|X|-2\sum_{i=1}^{k-1}\chi^{M_i}(X)}{(r^2-r-2)k-(r^2-3r-2)}.
\end{equation}
Thus for $v\in V$, since $\sum_{i=1}^{k-1}\chi^{M_i}(\partial (\{v\}))=k-1$,
we have $w_k(\partial (\{v\}))=\frac{[(r-1)k-(r-3)]r-2(k-1)}{(r^2-r-2)k-(r^2-3r-2)}=1$.
Finally, let $S\subseteq V$ with odd cardinality. Since $G$ is an $r$-graph, $|\partial (S)|\geq r$.
On the other hand, by recalling that $w_k(e)>\frac{1}{r+4}$ for each edge $e$, we have $w_k(\partial (S))>1$ provided by $|\partial (S)|\geq r+4$.
Hence, we may next assume that either $|\partial (S)|=r$ or $|\partial (S)|=r+2$ by parity.
In the former case, the formula $(\ref{eqn_r cut})$ implies $\sum_{i=1}^{k-1}\chi^{M_i}(\partial (S))=k-1$, and thus we can calculate from the formula $(\ref{b_eq_w_i(X)})$ that $w_k(\partial (S))=1$.
In the latter case, the formula $(\ref{eqn_r+2 cut})$ implies $\sum_{i=1}^{k-1}\chi^{M_i}(\partial (S))\leq r(k-1)+2$
and similarly, we get $w_k(\partial (S))\geq \frac{[(r-1)k-(r-3)](r+2)-2[r(k-1)+2]}{(r^2-r-2)k-(r^2-3r-2)}=1$.
This proves the claim.

By Lemma \ref{lem_FPMtoPM}, the graph $G$ has a 1-factor $M_k$ such that $$(1-\chi^{\bigcup_{i=1}^{k-1}M_i})\cdot \chi^{M_k} \geq (1-\chi^{\bigcup_{i=1}^{k-1}M_i})\cdot w_k.$$
Since the left side is just
$|\bigcup_{i=1}^{k}M_i|-|\bigcup_{i=1}^{k-1}M_i|$
and the right side equals to $\frac{(r-1)k-(r-3)}{(r^2-r-2)k-(r^2-3r-2)}(|E|-|\bigcup_{i=1}^{k-1}M_i|),$
it follows that
$$|\bigcup_{i=1}^kM_i|\geq \frac{(r-1)k-(r-3)}{(r^2-r-2)k-(r^2-3r-2)}|E|+\frac{(r^2-2r-1)k-(r^2-4r+1)}{(r^2-r-2)k-(r^2-3r-2)}|\bigcup_{i=1}^{k-1}M_i|,$$
which leads to
$$\frac{|\bigcup_{i=1}^k M_i|}{|E|}\geq 1-\prod_{i=1}^k \frac{(r^2-2r-1)i-(r^2-4r+1)}{(r^2-r-2)i-(r^2-3r-2)},$$
as desired.

Moreover, let $C$ be an edge cut of cardinality $r$.
The formula $(\ref{eqn_r cut})$ implies
$\sum_{i=1}^{k-1}\chi^{M_i}(C)=k-1$. On the other hand,
We can calculate from the formula (\ref{b_eq_w_i(X)}) that $w_k(C)=1$,
and thus $\chi^{M_k}(C)=1$ by Lemma \ref{lem_FPMtoPM}.
Therefore, $\sum_{i=1}^{k}\chi^{M_i}(C)=k$, as desired.

We next let $D$ be an edge cut of cardinality $r+2$.
Clearly, $\chi^{M_k} (D)\leq r+2$.
Thus if $\sum_{i=1}^{k-1}\chi^{M_i}(D)\leq r(k-1)$, then
$\sum_{i=1}^k\chi^{M_i}(D)\leq rk+2$, as desired.
By the formula $(\ref{eqn_r+2 cut})$ and by parity, we may next assume that $\sum_{i=1}^{k-1}\chi^{M_i}(D)=r(k-1)+2$.
By calculation we can get $w_k(D)=1$, and thus $\chi^{M_k}(D)=1$ by Lemma \ref{lem_FPMtoPM}, which also yields $\sum_{i=1}^k\chi^{M_i}(D)\leq rk+2$.
This completes the proof of this theorem.
\end{proof}

The following corollary is a direct consequence of this theorem.
\begin{corollary}\label{cor_lower bound}
	Let $r$ and $k$ be two positive integers with $r\geq 3$.
	If $r$ is even then
	$$m(r,k)\geq 1-\prod_{i=1}^k \frac{(r^2-3r+1)i-(r^2-5r+3)}{(r^2-2r-1)i-(r^2-4r-1)},$$
	and if $r$ is odd then
	$$m(r,k)\geq 1-\prod_{i=1}^k \frac{(r^2-2r-1)i-(r^2-4r+1)}{(r^2-r-2)i-(r^2-3r-2)}.$$
\end{corollary}

\section{Overfull graphs}\label{Overfull graphs}

We start with the following observations.

\begin{observation} \label{char_r_graphs}
	Let $r \geq 2$ be an integer. Every $r$-overfull-free $r$-regular graph is an $r$-graph. 
\end{observation}

\begin{observation} \label{obs:bipartite}
	A graph $G$ is 2-overfull-free if and only if
	$G$ is bipartite. 
\end{observation}

If $G$ is a graph, then $o(G)$ denotes the number of odd components of $G$. We will use the following theorem of Tutte.

\begin{theorem} [\cite{Tutte_1947}] \label{Tutte_1947}
	A graph $G$ has a 1-factor if and only of $o(G-S) \leq |S|$ for all $S \subseteq V(G)$.
\end{theorem}

\begin{proposition} \label{basic}
Let $k \geq 2$. If $G$ is a $k$-overfull graph, then $0 \leq s_k(G) \leq k-2$ and $k < \frac{|V(G)|}{|V(G)|-1} \Delta(G)$.   
\end{proposition}

\begin{proof}
Since $G$ is $k$-overfull, $2|E(G)|> k(|V(G)|-1)$ and $|V(G)|$ is odd. Notice that the two sides of this inequality has the same parity.
So, $2|E(G)|\geq k(|V(G)|-1)+2$, that is, $s_k(G) \leq k-2$.
Moreover,  by Handshaking Lemma, $2|E(G)|=\sum_{v\in V(G)}d_G(v)\leq \Delta(G)|V(G)|$.
Combining it with the fact that $2|E(G)|> k(|V(G)|-1)$, we deduce that $k < \frac{|V(G)|}{|V(G)|-1} \Delta(G)$.
\end{proof}

\begin{theorem} \label{structure_overfull}
	Let $k \geq 3$ be an integer. Every $k$-overfull graph contains a $(k-1)$-overfull subgraph. 
\end{theorem}

\begin{proof}
Suppose to the contrary that the statement is not true. Then there is a $k$-overfull graph $G$ which does not contain a $(k-1)$-overfull
subgraph. We may assume that $|V(G)|$ is minimum and according to this property $|E(G)|$ is minimum as well. 

It holds $\Delta(G) = k$, since for otherwise $G$ is $(k-1)$-overfull as well, a contradiction.

\begin{claim} \label{no_small_overfull_sub}
Let $H$ be a proper subgraph of $G$. If $H$ is of odd order, then $s_k(H) \geq k$.
\end{claim}

By the minimality of $G$, the subgraph $H$ is not $k$-overfull. Note that $H$ is of odd order and has maximum degree at most $k$. Thus, $\frac{2|E(H)|}{|V(H)|-1} \leq k$ and therefore,
$s_k(H) = k|V(H)| - 2|E(H)| \geq k|V(H)| - k|V(H)| + k = k$. 

\begin{claim}  \label{deficiency}
	$s_k(G) = k-2$, that is, $2|E(G)|= k(|V(G)|-1)+2$.
\end{claim}

Choose any edge $e$ of $G$. 
By Claim \ref{no_small_overfull_sub}, $s_k(G-e) \geq k$. It follows that $s_k(G)=s_k(G-e)-2 \geq k-2$. On the other hand, $s_k(G) \leq k-2$ by Proposition \ref{basic}.
Therefore, 	$s_k(G) = k-2$.

\begin{claim} \label{1-factor}
For every $z \in V(G)$, the graph $G-z$ has a 1-factor.
\end{claim}
	
Let $G'=G-z$. Then $s_k(G') = s_k(G) + d_G(z) - (k-d_G(z)) = k-2 + 2d_G(z) -k = 2d_G(z) - 2 \leq 2k-2$. 

Suppose to the contrary that $G'$ does not have a 
1-factor. By Theorem \ref{Tutte_1947}, there is $S \subseteq V(G')$ such that $o(G'-S) > |S|$. Let $O_1, \dots,O_n$ be the 
odd components of $G'-S$. Since $G-z$ has even order, $n$ and $|S|$ have the same parity. Thus, $n\geq |S|+2$.
	
With Claim \ref{no_small_overfull_sub} it follows that $s_k(O_i) \geq k$. Hence,
$|\partial_G(S)| \geq \sum_{i=1}^n s_k(O_i) - s_k(G') \geq nk-2k+2 = k(n-2) + 2 \geq k|S|+2$, a contradiction. 
	
\bigskip
We now deduce the statement. If $G$ is regular, then $s_k(G)=0=k-2$. So $k=2$, a contradiction. Hence,
there is $z \in V(G)$ such that $d_G(z) < k$. By Claim \ref{1-factor}, $G-z$ has a 1-factor $F$. 
Let $G'=G-F$. Then $\Delta(G') = k-1$, $|E(G')| = |E(G)| - \frac{1}{2}(|V(G)|-1)$, and 
$|V(G')|=|V(G)|$. Hence,
$\frac{2|E(G')|}{|V(G')|-1} = \frac{2|E(G)| - (|V(G)|-1)}{|V(G)|-1}
= \frac{2|E(G)|}{|V(G)|-1} - 1 > k-1$. This contradicts our assumption that $G$ does not contain a $(k-1)$-overfull
subgraph and the statement is proved.
\end{proof}

The following corollaries are immediate consequences of Theorem \ref{structure_overfull}.
The first one has the same flavor as a result of Vizing \cite{Vizing_1965} that a class 2 graph with chromatic index $k$ contains 
critical subgraphs with chromatic index $t$ for every $t \in \{2, \dots, k\}$.

\begin{corollary} \label{result_down}
	Let $k \geq 2$ be an integer and $G$ be a graph. If $G$ is $k$-overfull, 
	then $G$ contains a $t$-overfull subgraph for every $t \in \{2, \dots, k\}$.
\end{corollary}

\begin{corollary} \label{result_up}
	Let $k \geq 2$ be an integer and $G$ be a graph. If $G$ is $k$-overfull-free,
	then $G$ is $t$-overfull-free for every $t\geq k$.
\end{corollary}

\begin{corollary} \label{cor:decomposition}
Let $2 \leq k \leq r$ be integers and $G$ be an $r$-regular graph. If $G$ is $k$-overfull-free, then
$G$ is an $r$-graph and $G$ can be decomposed into a $(r-k)$-graph that is class 1 and a $k$-graph.  
\end{corollary}

\begin{proof}
By Corollary \ref{result_up}, $G$ is $r$-overfull-free and further, by Observation \ref{char_r_graphs}, $G$ is an $r$-graph. 
Let $F_1$ be a 1-factor of $G$. Consider $G-F_1$. If $k=r$, then we are done.
Hence, we may assume $k\leq r-1$. Similarly, we can deduce that $G-F_1$ is an $(r-1)$-graph having a 1-factor $F_2$.
Continue as
above till $G'=G-\bigcup_{i=1}^{r-k}F_i$. Then $G'$ and $G''=(V(G), \bigcup_{i=1}^{r-k}F_i)$
is the desired decomposition. 
\end{proof}

Corollary \ref{cor:decomposition} gives a sufficient condition for an $r$-graph decomposable into a $r_1$-graph and a $r_2$-graph for some $r_1$ and $r_2$. It also shows that for any $t$-overfull-free $r$-graph with $2\leq t \leq r$, we can obtain a better lower bound of $m(r,k,G)$ than the one of $m(r,k)$. 
More precisely, for $k\leq r-t$, take $k$ pairwise disjoint 1-factors of the class 1 graph from the decomposition by Corollary \ref{cor:decomposition}, which gives $m(r,k,G) = \frac{k}{r}$. For $k> r-t$, applying Theorem \ref{thm:main1} to the $t$-graph from the decomposition by Corollary \ref{cor:decomposition} gives $k$ 1-factors, which together with any $r-t$ many pairwise disjoint 1-factors of the class 1 graph from the decomposition leads to a better lower bound for $(m,k,G)$.

Moreover, Corollary \ref{cor:decomposition} confirms the following classical result.
\begin{theorem} 
	Let $k\geq 0$ be an integer. Every $k$-regular bipartite graph is class 1.
\end{theorem}
\begin{proof} For $k\in\{0,1\}$, the proof is trivial. For $k\geq 2$,
	let $G$ be a $k$-regular bipartite graph. By Observation \ref{obs:bipartite},  $G$ is 2-overfull-free. By Corollary \ref{cor:decomposition}, $G$ is decomposable into a class 1 subgraph and a 1-factor. Thus, $G$ is class 1.
\end{proof}

\section{Acknowledgment}
The authors are grateful to Giuseppe Mazzuoccolo for his useful comments and suggestions on 1-factor covering.
The authors are also grateful to two anonymous referees for their very careful reading and helpful comments. 

\end{document}